\documentclass[10pt]{article}

\usepackage{amsmath,amsthm}
\usepackage{amssymb}
\usepackage{hyperref}
\usepackage{lscape} 

\begin{document}

\title{On several kinds of sums of balancing numbers} 
\author{
Takao Komatsu\\
\small School of Science, Department of Mathematical Sciences\\[-0.8ex]
\small Zhejiang Sci-Tech University\\[-0.8ex]
\small Hangzhou 310018 China\\[-0.8ex]
\small \texttt{komatsu@zstu.edu.cn}\\\\
Gopal Krishna Panda\\
\small National Institute of Technology\\[-0.8ex]
\small Rourkela, India\\[-0.8ex]
\small \texttt{gkpanda\_nit@rediffmail.com} 
}

\date{\small 
}

\maketitle

\def\fl#1{\left\lfloor#1\right\rfloor}
\def\cl#1{\left\lceil#1\right\rceil}
\def\stf#1#2{\left[#1\atop#2\right]} 
\def\sts#1#2{\left\{#1\atop#2\right\}}

\newtheorem{theorem}{Theorem}
\newtheorem{Prop}{Proposition}
\newtheorem{Cor}{Corollary}
\newtheorem{Lem}{Lemma}

\begin{abstract}  
The balancing numbers $B_n$ ($n=0,1,\cdots$) are solutions of the binary recurrence $B_n=6B_{n-1}-B_{n-2}$ ($n\ge 2$) with $B_0=0$ and $B_1=1$. 
In this paper we show several relations about the sums of product of two balancing numbers of the type $\sum_{m=0}^n B_{k m+r}B_{k(n-m)+r}$ ($k>r\ge 0$) and the alternating sum of reciprocal of  balancing numbers $\fl{\left(\sum_{k=n}^\infty\frac{1}{B_{l k}}\right)^{-1}}$. Similar results are also obtained for Lucas-balancing numbers $C_n$ ($n=0,1,\cdots$), satisfying the binary recurrence $C_n=6C_{n-1}-C_{n-2}$ ($n\ge 2$) with $C_0=1$ and $C_1=3$.  Some binomial sums involving these numbers are also explored.  \\
\end{abstract}

\section{Introduction}
 
A positive integer $x$ is called {\it balancing number} if  
\begin{equation}  
1+2+\cdots+(x-1)=(x+1)+\cdots+(y-1) 
\label{def:balancing}
\end{equation}  
holds for some integer $y\ge x+2$. The problem of determining all balancing numbers leads to a Pell equation, whose solutions in $x$ can be described by the recurrence $B_n=6B_{n-1}-B_{n-2}$ ($n\ge 2$) with $B_0=0$ and $B_1=1$ (see \cite{behera-panda,Finkelstein}).  
One of the most general extensions of balancing numbers is when (\ref{def:balancing}) is being replaced by   
\begin{equation}  
1^k+2^k+\cdots+(x-1)^k=(x+1)^l+\cdots+(y-1)^l\,,  
\label{kpower-balancing}
\end{equation}  
where the exponents $k$ and $l$ are given positive integers. In the work of Liptai et al. \cite{LLPS} effective and non-effective finiteness theorems on (\ref{kpower-balancing}) are proved.  In \cite{KS2014} a balancing problem of ordinary binomial coefficients is studied.  

In addition, $C_n$ denotes the $n$-th {\it Lucas-balancing number}, satisfying 
 $C_n=6C_{n-1}-C_{n-2}$ ($n\ge 2$) with $C_0=1$ and $C_1=3$.

\section{Properties of balancing numbers}  

In this section, we describe several basic properties about balancing numbers and Lucas-balancing numbers.  
$B_n$ and $C_n$ are expressed as 
$$
B_n=\frac{\alpha^n-\beta^n}{4\sqrt{2}}\quad\hbox{and}\quad C_n=\frac{\alpha^n+\beta^n}{2}\quad(n\ge 0)\,,
$$ 
where $\alpha=3+2\sqrt{2}$ and $\beta=3-2\sqrt{2}$.

The balancing and Lucas-balancing numbers satisfy the identities 
\begin{equation} 
B_{n-r}B_{n+r}=B_n^2-B_r^2,\quad\hbox{and}\quad 
C_{n-r}C_{n+r}=C_n^2+C_r^2-1
\quad(n\ge r)\,, 
\label{relation1}
\end{equation} 
respectively (\cite{ray2009}). In particular, by putting $r=1$,  we have
\begin{equation} 
B_{n-1}B_{n+1}=B_n^2-1,\quad\hbox{and}\quad 
C_{n-1}C_{n+1}=C_n^2+8
\quad(n\ge 1)\,, 
\label{relation2}
\end{equation} 
More general relations for more generalized numbers can be seen in \cite[Theorem 2.14, Theorem 2.16]{SK}.  
In \cite{KK}, several properties of balancing and cobalancing numbers, square, triangular, and oblong numbers, and other numbers are derived. Their main results are that if $B_n,B_m>1$, then $B_n B_m$ is not a balancing number.  

The identity 
$$ 
B_1+B_3+\cdots+B_{2n-1}=B_n^2 
$$ 
gives  
$$ 
B_{2n-1}=B_n^2-B_{n-1}^2\,. 
$$  
In general, for $a,b\ge 0$ we have 
$$ 
B_{a+b+1}=B_{a+1}B_{b+1}-B_{a}B_{b}\,. 
$$ 

There are some relations between $B_n$ and $C_n$, including 
\begin{equation}  
C_n=\sqrt{8 B_n^2+1}
\label{eq:c-b}
\end{equation} 
(\cite{panda2009}).

In \cite{panda2009}, several basic relations, which are analogous of $\sin(m\pm n)$, $\cos(m\pm n)$, $\sin(-n)$ and $\cos(-n)$ are given.  
\begin{Lem}  
For integers $m$ and $n$ with $m\ge n$ 
\begin{enumerate}  
\item $B_{n\pm m}=B_n C_m\pm B_m C_n$ 
\item $C_{n\pm m}=C_n C_m\pm 8 B_m B_n$  
\item $B_{-n}=-B_n$ 
\item $C_{-n}=C_n$ 
\end{enumerate}  
\label{lem2009} 
\end{Lem} 

From Lemma \ref{lem2009}, we can easily get the following relations, which are used in the next section.  

\begin{Prop}
If $m$ and $n$ are natural numbers, then  
\begin{enumerate}  
\item $B_{n+m}-2 B_n C_m=
\begin{cases} 
-B_{n-m}&\text{if $n\ge m$};\\
B_{m-n}&\text{if $n<m$}. 
\end{cases}$  
\item $C_{n+m}-2 C_n C_m=
\begin{cases} 
-C_{n-m}&\text{if $n\ge m$};\\
-C_{m-n}&\text{if $n<m$}. 
\end{cases}$  
\end{enumerate} 
\label{prp2009} 
\end{Prop} 

In \cite{ray2012}, several interesting identities are given for balancing $B_n$ and Lucas-balancing numbers $C_n$.  Some of them are listed below. $\left(\frac{p}{8}\right)$ denotes the Legendre symbol. 

\begin{Prop} 
Let $m$ and $n$ be positive integers. Then 
\begin{enumerate}  
\item $\gcd(B_m,B_n)=B_{\gcd(m,n)}$. 
\item If $p$ is an odd prime, then $C_p\equiv 3\pmod p$ and $B_p\equiv\left(\frac{p}{8}\right)\pmod p$.  
\item For any positive integer $m$, $B_{2 m}\equiv 0\pmod{C_m}$, $B_{2 m-1}\equiv 1\pmod{C_{m}}$.  
\end{enumerate}  
\label{prp2012} 
\end{Prop}

\begin{Prop}  
For $n\ge 0$, 
$$
\sum_{k=0}^n\binom{n}{k}(-1)^{n-k}3^k B_k=
\begin{cases} 
2^{\frac{3 n}{2}}B_n&\text{if $n$ is even};\\
2^{\frac{3(n-1)}{2}}C_n&\text{if $n$ is odd}  
\end{cases} 
$$ 
and 
$$
\sum_{k=0}^n\binom{n}{k}(-1)^{n-k}3^k C_k=
\begin{cases} 
2^{\frac{3 n}{2}}C_n&\text{if $n$ is even};\\
2^{\frac{3(n+1)}{2}}B_n&\text{if $n$ is odd}\,. 
\end{cases} 
$$ 
\label{prp1}  
\end{Prop}  
\begin{proof}  
We prove the first part.  The second part is proven similarly and omitted.  
Notice that $3\alpha-1=2\sqrt{2}\alpha$ and $3\beta-1=-2\sqrt{2}$.  
When $n$ is even, the left-hand side is equal to 
\begin{align*}  
\frac{1}{4\sqrt{2}}\sum_{k=0}^n\binom{n}{k}(-1)^{n-k}3^k(\alpha^k-\beta^k)
&=\frac{(3\alpha-1)^n-(3\beta-1)^n}{4\sqrt{2}}\\
&=\frac{(2\sqrt{2}\alpha)^n-(2\sqrt{2}\beta)^n}{4\sqrt{2}}\\
&=(2\sqrt{2})^n B_n\,. 
\end{align*} 
When $n$ is odd, the left-hand side is equal to 
\begin{align*}  
\frac{1}{4\sqrt{2}}\sum_{k=0}^n\binom{n}{k}(-1)^{n-k}3^k(\alpha^k-\beta^k)
&=\frac{(3\alpha-1)^n-(3\beta-1)^n}{4\sqrt{2}}\\
&=\frac{(2\sqrt{2}\alpha)^n+(2\sqrt{2}\beta)^n}{4\sqrt{2}}\\
&=\frac{(2\sqrt{2})^n C_n}{2\sqrt{2}}\,. 
\end{align*} 
\end{proof}

\begin{Prop}  
For $n\ge 0$, 
\begin{align}  
\sum_{k=0}^{2 n}\binom{2 n}{k}B_k&=8^n B_n\,, 
\label{prp2b1}\\ 
\sum_{k=0}^{2 n}\binom{2 n}{k}C_k&=8^n C_n\,, 
\label{prp2c1}\\
\sum_{k=0}^{2 n}\binom{2 n}{k}(-1)^k B_k&=4^n B_n\,, 
\label{prp2b2}\\
\sum_{k=0}^{2 n}\binom{2 n}{k}(-1)^k C_k&=4^n C_n\,.  
\label{prp2c2}    
\end{align}  
\label{prp2}
\end{Prop} 
\begin{proof}  
By using $(\alpha+1)^2=8\alpha$ and $(\beta+1)^2=8\beta$, we get (\ref{prp2b1}) and (\ref{prp2c1}), respectively.   
By using $(\alpha-1)^2=4\alpha$ and $(\beta-1)^2=4\beta$, we get (\ref{prp2b2}) and (\ref{prp2c2}), respectively.  
\end{proof}

\section{The sums of products of two balancing numbers}  

In this section, we consider the sums of products of two balancing numbers.  The sums of products of various numbers with or without binomial coefficients have been studied (e.g., \cite{AD1,Dilcher,Kamano1,Kamano2,Komatsu3,Komatsu4,Komatsu2014,Komatsu2014b,Komatsu2014,Komatsu2015,Komatsu6,KS,Zhao}).  One of the famous results are about Bernoulli numbers $\mathfrak B_n$, defined by the generating function 
$$
\frac{t}{e^t-1}=\sum_{n=0}^\infty\mathfrak B_n\frac{t^n}{n!}\,. 
$$  
The following identity (with binomial coefficients) on sums of products of two Bernoulli numbers is known as Euler's formula: 
$$
\sum_{i=0}^n\binom{n}{i}\mathfrak B_i\mathfrak B_{n-i}=-\mathfrak B_{n-1}-(n-1)\mathfrak B_n\quad(n\ge 0)\,. 
$$ 

The new results (without binomial coefficients) are based upon the generating function of balancing numbers of the types $B_{k n+r}$ and those of Lucas-balancing numbers $C_{k n+r}$, which are obtained in (\ref{gf:b}) and (\ref{gf:c}), respectively.  Notice that the corresponding generating functions of Fibonacci numbers and Lucas numbers are given by 
\begin{equation}  
\frac{F_r+(-1)^r F_{k-r}t}{1-L_k t+(-1)^k 
t^2}=\sum_{n=0}^\infty F_{k n+r}t^n
\label{fib:gen}
\end{equation}  
and 
\begin{equation} 
\frac{L_r+(-1)^r L_{k-r}t}{1-L_k t+(-1)^k t^2}=\sum_{n=0}^\infty L_{k n+r}t^n\,,
\label{luca:gen}
\end{equation}  
respectively (\cite[p.230]{Koshy}) .

\begin{theorem}  
Let $k$ and $r$ be fixed integers with $k>r\ge 0$. Then we have 
\begin{multline*} 
\sum_{m=0}^n B_{k m+r}B_{k(n-m)+r}=B_{k-r}\biggl(-(n+1)B_{k(n+1)+r}\\
\quad\left.+\sum_{j=0}^n\frac{B_k B_{k-r}^j}{2}\left(\frac{(-1)^j}{(B_k+B_r)^{j+1}}+\frac{1}{(B_k-B_r)^{j+1}}\right)(n-j+1)B_{k(n-j+1)+r}\right) 
\end{multline*} 
and 
\begin{multline*} 
\sum_{m=0}^n C_{k m+r}C_{k(n-m)+r}=C_{k-r}\biggl((n+1)C_{k(n+1)+r}\\
-\sum_{j=0}^n\frac{2\sqrt{2}B_k(C_{k-r}\sqrt{-1})^j}{2}\biggl(\frac{1}{(2\sqrt{2}B_k+C_r\sqrt{-1})^{j+1}}\\
\left.+\frac{(-1)^j}{(2\sqrt{2}B_k-C_r\sqrt{-1})^{j+1}}\biggr)(n-j+1)C_{k(n-j+1)+r}\right)\,. 
\end{multline*} 
\label{th:convo}
\end{theorem} 

\begin{proof}
From the first identity of Proposition \ref{prp2009}, we get $B_{m k+r}-2 C_k B_{(m-1)k+r}=-B_{(m-2)k+r}$ and $B_{k+r}-2 C_k B_r=B_{k-r}$. Hence, the generating function of $B_{k n+r}$ is given by 
\begin{equation} 
b(t):=
\frac{B_r+B_{k-r}t}{1-2 C_k t+t^2}=\sum_{n=0}^\infty B_{k n+r}t^n\,. 
\label{gf:b}
\end{equation}  
Thus, by using (\ref{relation1}), we have 
\begin{align*}  
b(t)^2&=\frac{(B_r+B_{k-r}t)^2}{B_{k+r}-2 B_r t-B_{k-r}t^2}b'(t)\\
&=\left(-B_{k-r}+\frac{B_r^2+B_{k+r}B_{k-r}}{B_{k+r}-2 B_r t-B_{k-r}t^2}\right)b'(t)\\
&=\left(-B_{k-r}+\frac{1}{B_{k-r}}\frac{B_k^2}{\left(\frac{B_k}{B_{k-r}}\right)^2-\left(\frac{B_r}{B_{k-r}}+t\right)^2}\right)b'(t)\\
&=B_{k-r}\left(-1+\sum_{l=0}^\infty\left(\frac{B_r}{B_k}+\frac{B_{k-r}}{B_k}t\right)^{2 l}\right)b'(t)\\
&=B_{k-r}\left(-1+\sum_{l=0}^\infty\sum_{j=0}^{2 l}\binom{2 l}{j}\left(\frac{B_r}{B_k}\right)^{2 l-j}\left(\frac{B_{k-r}}{B_k}\right)^j t^j\right)b'(t)\,.
\end{align*}  
Since 
$$
b'(t)=\sum_{n=0}^\infty(n+1)B_{k(n+1)+r}t^n\,, 
$$ 
we obtain 
\begin{align}  
b(t)^2&=B_{k-r}\left(-\sum_{n=0}^\infty(n+1)B_{k(n+1)+r}t^n\right.\notag\\
&\quad\left.+\sum_{l=0}^\infty\sum_{j=0}^{2 l}\binom{2 l}{j}\left(\frac{B_r}{B_k}\right)^{2 l-j}\left(\frac{B_{k-r}}{B_k}\right)^j\sum_{n=j}^\infty(n-j+1)B_{k(n-j+1)+r}t^n\right)
\label{1234}\\
&=B_{k-r}\sum_{n=0}^\infty\biggl(-(n+1)B_{k(n+1)+r}\notag\\
&\quad\left.+\sum_{j=0}^{n}\sum_{l=\fl{\frac{j+1}{2}}}^\infty\binom{2 l}{j}\left(\frac{B_r}{B_k}\right)^{2 l-j}\left(\frac{B_{k-r}}{B_k}\right)^j(n-j+1)B_{k(n-j+1)+r}\right)t^n\,. \notag 
\end{align}  
Since for a positive integer $s$ and a real number $x$ with $|x|<1$ 
$$
\frac{(1+x)^{-s}+(1-x)^{-s}}{2}=\binom{s-1}{s-1}+\binom{s+1}{s-1}x^2+\binom{s+3}{s-1}x^4+\cdots 
$$ 
and 
$$
\frac{-(1+x)^{-s}+(1-x)^{-s}}{2}=\binom{s}{s-1}x+\binom{s+2}{s-1}x^3+\binom{s+4}{s-1}x^5+\cdots\,, 
$$ 
we get 
\begin{align*} 
&\sum_{l=\fl{\frac{j+1}{2}}}^\infty\binom{2 l}{j}\left(\frac{B_r}{B_k}\right)^{2 l-j}\left(\frac{B_{k-r}}{B_k}\right)^j\\ 
&=\frac{1}{2}\left((-1)^j\left(1+\frac{B_r}{B_k}\right)^{-j-1}+\left(1-\frac{B_r}{B_k}\right)^{-j-1}\right)\left(\frac{B_{k-r}}{B_k}\right)^j\\
&=\frac{B_k B_{k-r}^j}{2}\left(\frac{(-1)^j}{(B_k+B_r)^{j+1}}+\frac{1}{(B_k-B_r)^{j+1}}\right)\,. 
\end{align*} 
Since 
$$
b(t)^2=\sum_{n=0}^\infty\sum_{m=0}^n B_{k m+r}B_{k(n-m)+r}t^n\,, 
$$ 
by comparing the coefficients on both sides, we get the first identity.  

From the second identity of Proposition \ref{prp2009}, we get $C_{m k+r}-2 C_k C_{(m-1)k+r}=-C_{(m-2)k+r}$ and $C_{k+r}-2 C_k C_r=-C_{k-r}$. Hence, the generating function of $C_{k n+r}$ is given by 
\begin{equation} 
c(t):=\frac{C_r-C_{k-r}t}{1-2 C_k t+t^2}=\sum_{n=0}^\infty C_{k n+r}t^n\,. 
\label{gf:c}
\end{equation}  
Similarly, by using the relations (\ref{eq:c-b}) and (\ref{relation1}), we have 
\begin{align*}  
c(t)^2&=\frac{(C_r-C_{k-r}t)^2}{C_{k+r}-2 C_r t+C_{k-r}t^2}c'(t)\\
&=\left(C_{k-r}+\frac{C_r^2-C_{k+r}C_{k-r}}{C_{k+r}-2 C_r t+C_{k-r}t^2}\right)c'(t)\\
&=\left(C_{k-r}-\frac{1}{C_{k-r}}\frac{C_k^2-1}{\frac{C_k^2-1}{C_{k-r}^2}+\left(\frac{C_r}{C_{k-r}}-t\right)^2}\right)c'(t)\\
&=C_{k-r}\left(1-\sum_{l=0}^\infty(-1)^l\left(\frac{C_r}{2\sqrt{2}B_k}-\frac{C_{k-r}}{2\sqrt{2}B_k}t\right)^{2 l}\right)c'(t)\\
&=C_{k-r}\Biggl(1\\
&\quad-\sum_{l=0}^\infty(-1)^l\sum_{j=0}^{2 l}(-1)^j\binom{2 l}{j}\left(\frac{C_r}{2\sqrt{2}B_k}\right)^{2 l-j}\left(\frac{C_{k-r}}{2\sqrt{2}B_k}\right)^j t^j\Biggr)c'(t)\,.
\end{align*}  
Since 
$$
c'(t)=\sum_{n=0}^\infty(n+1)C_{k(n+1)+r}t^n\,, 
$$ 
we obtain 
\begin{align*}  
c(t)^2
&=C_{k-r}\biggl(\sum_{n=0}^\infty(n+1)C_{k(n+1)+r}t^n\\
&\quad-\sum_{l=0}^\infty(-1)^l\sum_{j=0}^{2 l}(-1)^j\binom{2 l}{j}\left(\frac{C_r}{2\sqrt{2}B_k}\right)^{2 l-j}\left(\frac{C_{k-r}}{2\sqrt{2}B_k}\right)^j\\
&\qquad\qquad\times\sum_{n=j}^\infty(n-j+1)C_{k(n-j+1)+r}t^n\biggr)\\
&=C_{k-r}\sum_{n=0}^\infty\biggl((n+1)C_{k(n+1)+r}\\
&\quad-\sum_{j=0}^{n+1}(-1)^j\sum_{l=\fl{\frac{j+1}{2}}}^\infty(-1)^l\binom{2 l}{j}\left(\frac{C_r}{2\sqrt{2}B_k}\right)^{2 l-j}\left(\frac{C_{k-r}}{2\sqrt{2}B_k}\right)^j\\
&\qquad\qquad\times(n-j+1)C_{k(n-j+1)+r}\biggr)t^n\,. 
\end{align*} 
Since for a positive integer $s$ and a real number $x$ with $|x|<1$ 
$$
\frac{(1+x\sqrt{-1})^{-s}+(1-x\sqrt{-1})^{-s}}{2}=\binom{s-1}{s-1}-\binom{s+1}{s-1}x^2+\binom{s+3}{s-1}x^4-\cdots 
$$ 
and 
\begin{multline*} 
\frac{-(1+x\sqrt{-1})^{-s}+(1-x\sqrt{-1})^{-s}}{2}\\
=\sqrt{-1}\left(\binom{s}{s-1}x-\binom{s+2}{s-1}x^3+\binom{s+4}{s-1}x^5-\cdots\right)\,, 
\end{multline*}  
we get 
\begin{align*} 
&\sum_{l=\fl{\frac{j+1}{2}}}^\infty(-1)^l\binom{2 l}{j}\left(\frac{C_r}{2\sqrt{2}B_k}\right)^{2 l-j}\left(\frac{C_{k-r}}{2\sqrt{2}B_k}\right)^j\\ 
&=\frac{1}{2}\left((-1)^j\left(1+\frac{C_r}{2\sqrt{2}B_k}\right)^{-j-1}+\left(1-\frac{C_r}{2\sqrt{2}B_k}\right)^{-j-1}\right)\left(\frac{C_{k-r}}{2\sqrt{2}B_k}\right)^j\\
&=\frac{2\sqrt{2}B_k(C_{k-r}\sqrt{-1})^j}{2}\\
&\qquad\times\left(\frac{(-1)^j}{(2\sqrt{2}B_k+C_r\sqrt{-1})^{j+1}}+\frac{1}{(2\sqrt{2}B_k-C_r\sqrt{-1})^{j+1}}\right)\,. 
\end{align*} 
Since 
$$
c(t)^2=\sum_{n=0}^\infty\sum_{m=0}^n C_{k m+r}C_{k(n-m)+r}t^n\,, 
$$ 
by comparing the coefficients on both sides, we get the second identity.  
\end{proof}

If $r=0$ in the first identity of Theorem \ref{th:convo}, by putting $r=0$ in (\ref{1234}), the formula can be simplified because $B_0=0$. 

\begin{Cor} 
For a positive integer $k$, we have 
$$  
\sum_{m=0}^n B_{k m}B_{k(n-m)}=B_{k}\sum_{l=1}^{\fl{\frac{n+1}{2}}}(n-2 l+1)B_{k(n-2 l+1)}\,. 
$$ 
\end{Cor}

This method is also applicable to Fibonacci and Lucas numbers.   Similarly to Theorem \ref{th:convo}, by (\ref{fib:gen}) with 
$$
(-1)^r F_{k-r}+F_r L_k=F_{k+r}\quad\hbox{and}\quad 
(-1)^{k-r}F_{k+r}F_{k-r}+F_r^2=(-1)^{k-r}F_k^2\,, 
$$ 
we have the relation for Fibonacci numbers.  

\begin{theorem}  
Let $k$ and $r$ be fixed integers with $k>r\ge 0$. If $k-r$ is even,  then 
\begin{multline*} 
\sum_{m=0}^n F_{k m+r}F_{k(n-m)+r}=F_{k-r}\biggl(-(n+1)F_{k(n+1)+r}\\
\quad\left.+\sum_{j=0}^n(-1)^{r j}\frac{F_k F_{k-r}^j}{2}\left(\frac{(-1)^j}{(F_k+F_r)^{j+1}}+\frac{1}{(F_k-F_r)^{j+1}}\right)(n-j+1)F_{k(n-j+1)+r}\right)\,. 
\end{multline*} 
If $k-r$ is odd, then 
\begin{multline*} 
\sum_{m=0}^n F_{k m+r}F_{k(n-m)+r}=F_{k-r}\biggl((n+1)F_{k(n+1)+r}\\
-\sum_{j=0}^n(-1)^{r j}\frac{F_k(F_{k-r}\sqrt{-1})^j}{2}\biggl(\frac{1}{(F_k+F_r\sqrt{-1})^{j+1}}\\
\left.+\frac{(-1)^j}{(F_k-F_r\sqrt{-1})^{j+1}}\biggr)(n-j+1)F_{k(n-j+1)+r}\right)\,. 
\end{multline*} 
\label{th:fibconvo}
\end{theorem} 

By (\ref{luca:gen}) with 
$$
(-1)^{r-1} L_{k-r}+L_k L_r=L_{k+r}
$$
and 
$$ 
(-1)^{k-r}L_{k+r}L_{k-r}-L_r^2=(-1)^{k-r}L_k^2-4(-1)^r=(-1)^{k-r}5 F_k^2\,, 
$$ 
we have the relation for Lucas numbers.  

\begin{theorem}  
Let $k$ and $r$ be fixed integers with $k>r\ge 0$. If $k-r$ is even,  then 
\begin{multline*} 
\sum_{m=0}^n L_{k m+r}L_{k(n-m)+r}=L_{k-r}\biggl((n+1)L_{k(n+1)+r}\\
-\sum_{j=0}^n(-1)^{(r+1)j}\frac{\sqrt{5}F_k(L_{k-r}\sqrt{-1})^j}{2}\biggl(\frac{(-1)^j}{(\sqrt{5}F_k+L_r\sqrt{-1})^{j+1}}\\
\left.+\frac{1}{(\sqrt{5}F_k-L_r\sqrt{-1})^{j+1}}\biggr)(n-j+1)L_{k(n-j+1)+r}\right)\,. 
\end{multline*} 
If $k-r$ is odd, then 
\begin{multline*} 
\sum_{m=0}^n L_{k m+r}L_{k(n-m)+r}=L_{k-r}\biggl(-(n+1)L_{k(n+1)+r}\\
+\sum_{j=0}^n(-1)^{(r+1)j}\frac{\sqrt{5}F_kL_{k-r}^j}{2}\biggl(\frac{(-1)^j}{(\sqrt{5}F_k+L_r)^{j+1}}\\
\left.+\frac{1}{(\sqrt{5}F_k-L_r)^{j+1}}\biggr)(n-j+1)L_{k(n-j+1)+r}\right)\,. 
\end{multline*} 
\label{th:lucaconvo}
\end{theorem}

\section{Reciprocal sums of balancing numbers}  

In this section we consider several results on reciprocal sums of Balancing numbers and Lucas-balancing numbers.  Here $\fl{x}$ denotes the integer part of a real number of $x$.  

Many authors studied bounds for reciprocal sums involving terms from Fibonacci and other related numbers (e.g., see \cite{Komatsu2010,Komatsu2011,Kuhapatanakul,Ohtsuka,Zhang}).   In \cite{Holliday}, several identities are shown for generalized Fibonacci numbers $G_n$, defined by 
$$
G_n=a G_{n-1}+G_{n-2}\quad(n\ge 2),\quad G_0=0,\quad G_1=1, 
$$  
where $a$ is a positive integer. 
Some of them are the following. 

\begin{Prop}\label{fib:relations}  
\begin{align*} 
&{\rm (1)}\qquad \fl{\left(\sum_{k=n}^\infty\frac{1}{G_k}\right)^{-1}}=
\begin{cases}
G_n-G_{n-1}&\text{if $n$ is even and $n\ge 2$};\\
G_n-G_{n-1}-1&\text{if $n$ is odd and $n\ge 1$}\,.
\end{cases}\\ 
&{\rm (2)}\qquad \fl{\left(\sum_{k=n}^\infty\frac{1}{G_k^2}\right)^{-1}}=
\begin{cases}
a G_{n-1}G_n-1&\text{if $n$ is even and $n\ge 2$};\\
a G_{n-1}G_n&\text{if $n$ is odd and $n\ge 1$}\,.
\end{cases}\\ 
&{\rm (3)}\qquad \fl{\left(\sum_{k=n}^\infty\frac{1}{G_{2k}}\right)^{-1}}=
G_{2n}-G_{2n-2}-1\qquad(n\ge 1)\\   
&{\rm (4)}\qquad \fl{\left(\sum_{k=n}^\infty\frac{1}{G_{2k-1}}\right)^{-1}}=
G_{2n-1}-G_{2n-3}\qquad(n\ge 2)\\  
\end{align*}   
\end{Prop}  

We provide several analogous results about alternating sums of reciprocal balancing numbers $B_n$ and Lucas-balancing numbers $C_n$.  

\begin{theorem}  
Let $l$ be a fixed positive number.  Then for $n\ge 1$ we have 
\begin{align} 
\fl{\left(\sum_{k=n}^\infty\frac{1}{B_{l k}}\right)^{-1}}&=B_{l n}-B_{l(n-1)}-1\,,\label{b-l}\\ 
\fl{\left(\sum_{k=n}^\infty\frac{1}{C_{l k}}\right)^{-1}}&=C_{l n}-C_{l(n-1)}\,.
\label{lb-lc}
\end{align} 
\label{b-lb}
\end{theorem}  
\begin{proof}  
By (\ref{relation1}), we have $B_{l n}^2=B_{l(n-1)}B_{l(n+1)}+B_l^2$.    
Since $B_{l(n+1)}-B_l^2>B_{l(n-1)}+1$ ($n\ge 1$), we have 
\begin{align*}  
\frac{1}{B_{l n}-B_{l(n-1)}-1}-\frac{1}{B_{l n}}&=\frac{1}{B_{l n}\left(\frac{B_{l n}}{B_{l(n-1)}+1}-1\right)}\\
&=\frac{1}{\frac{B_{l(n+1)}B_{l(n-1)}+B_l^2}{B_{l(n-1)}+1}-B_{l n}}\\
&=\frac{1}{B_{l(n+1)}-B_{l n}-\frac{B_{l(n+1)}-B_l^2}{B_{l(n-1)}+1}}\\
&>\frac{1}{B_{l(n+1)}-B_{l n}-1}\,. 
\end{align*}  
Thus, we have 
\begin{equation}  
\frac{1}{B_{l n}-B_{l(n-1)}-1}>\sum_{k=n}^\infty\frac{1}{B_{l k}}\,. 
\label{80} 
\end{equation}  
On the other hand, we have 
\begin{align*}  
\frac{1}{B_{l n}-B_{l(n-1)}}-\frac{1}{B_{l n}}&=\frac{1}{B_{l n}\left(\frac{B_{l n}}{B_{l(n-1)}}-1\right)}\\
&=\frac{1}{\frac{B_{l(n+1)}B_{l(n-1)}+B_l^2}{B_{l(n-1)}}-B_{l n}}\\
&=\frac{1}{B_{l(n+1)}-B_{l n}+\frac{B_l^2}{B_{l(n-1)}}}\\
&<\frac{1}{B_{l(n+1)}-B_{l n}}\,. 
\end{align*}  
Thus, we have 
\begin{equation}  
\sum_{k=n}^\infty\frac{1}{B_{l k}}>\frac{1}{B_{l n}-B_{l(n-1)}}\,. 
\label{81} 
\end{equation} 
Combining (\ref{80}) and (\ref{81}), we obtain the identity (\ref{b-l}).  
Similarly, by $C_{l n}^2=C_{l(n-1)}C_{l(n+1)}-C_l^2+1$, we obtain the identity (\ref{lb-lc}).  
\end{proof}

\begin{theorem}  
$$ 
\fl{\left(\sum_{k=n}^\infty\frac{(-1)^k}{B_k}\right)^{-1}}=
\begin{cases}  
B_n+B_{n-1}&\text{if $n$ is even};\\ 
-(B_n+B_{n-1}+1)&\text{if $n$ is odd}. 
\end{cases} 
$$ 
\end{theorem}  

\begin{proof}  
First, we prove that 
\begin{equation}  
\frac{1}{B_{n+2}+B_{n+1}}+\frac{1}{B_n}-\frac{1}{B_{n+1}}<\frac{1}{B_n+B_{n-1}}\,.
\label{eq10} 
\end{equation}  
Since $B_n^2-B_{n-1}B_{n+1}=1$ by (\ref{relation2}), we have  
\begin{align*}  
&\left(\frac{1}{B_{n+1}}-\frac{1}{B_{n+2}+B_{n+1}}\right)-\left(\frac{1}{B_n}-\frac{1}{B_n+B_{n-1}}\right)\\
&=\frac{B_{n+2}}{B_{n+1}(B_{n+2}+B_{n+1})}-\frac{B_{n-1}}{B_n(B_n+B_{n-1})}\\
&=\frac{1}{B_{n+1}\left(1+\frac{B_n}{B_{n+1}}+\frac{1}{B_{n+2}B_{n+1}}\right)}-\frac{1}{B_{n}\left(\frac{B_{n+1}}{B_{n}}+\frac{1}{B_{n}B_{n-1}}+1\right)}\\
&=\frac{1}{B_{n+1}+B_n+\frac{1}{B_{n+2}}}-\frac{1}{B_{n+1}+B_n+\frac{1}{B_{n-1}}}\\
&>0\,. 
\end{align*}  
Therefore, repeating (\ref{eq10}), we obtain 
\begin{equation} 
\frac{1}{B_n+B_{n-1}}>\sum_{i=0}^\infty\left(\frac{1}{B_{n+2i}}-\frac{1}{B_{n+2i+1}}\right)\,.  
\label{eq15} 
\end{equation}   
Second, we prove that 
\begin{equation}   
\frac{1}{B_n+B_{n-1}+1}<\frac{1}{B_n}-\frac{1}{B_{n+1}}+\frac{1}{B_{n+2}+B_{n+1}+1}\,. 
\label{eq20}
\end{equation}  
Since 
$$
\frac{B_{n+1}-1}{B_{n-1}+1}>1>\frac{B_n-1}{B_{n+2}+1}\,, 
$$ 
we have  
\begin{align*}  
&\left(\frac{1}{B_n}-\frac{1}{B_n+B_{n-1}+1}\right)-\left(\frac{1}{B_{n+1}}-\frac{1}{B_{n+2}+B_{n+1}+1}\right)\\
&=\frac{1}{B_n\left(\frac{B_n}{B_{n-1}+1}+1\right)}-\frac{1}{B_{n+1}\left(\frac{B_{n+1}}{B_{n+2}+1}+1\right)}\\
&=\frac{1}{\frac{B_{n-1}B_{n+1}+1}{B_{n-1}+1}+B_n}-\frac{1}{\frac{B_{n+2}B_{n}+1}{B_{n+2}+1}+B_{n+1}}\\
&=\frac{1}{B_{n+1}+B_n-\frac{B_{n+1}-1}{B_{n-1}+1}}-\frac{1}{B_{n+1}+B_n-\frac{B_{n}-1}{B_{n+2}+1}}\\
&>0\,. 
\end{align*}  
Therefore, repeating (\ref{eq20}), we obtain  
\begin{equation} 
\frac{1}{B_n+B_{n-1}+1}<\sum_{i=0}^\infty\left(\frac{1}{B_{n+2i}}-\frac{1}{B_{n+2i+1}}\right)\,.  
\label{eq25} 
\end{equation}  
Combining (\ref{eq15}) and (\ref{eq25}), we get 
$$
\frac{1}{B_n+B_{n-1}+1}<\sum_{k=n}^\infty\frac{(-1)^k}{B_k}<\frac{1}{B_n+B_{n-1}}
$$ 
if $n$ is even, and 
$$
-\frac{1}{B_n+B_{n-1}+1}>\sum_{k=n}^\infty\frac{(-1)^k}{B_k}>-\frac{1}{B_n+B_{n-1}}
$$ 
if $n$ is odd. 
\end{proof}

\begin{theorem}  
$$ 
\fl{\left(\sum_{k=n}^\infty\frac{(-1)^k}{B_k^2}\right)^{-1}}=
\begin{cases}  
B_n^2+B_{n-1}^2&\text{if $n$ is even};\\ 
-(B_n^2+B_{n-1}^2+1)&\text{if $n$ is odd}. 
\end{cases} 
$$ 
\end{theorem}  
\begin{proof} 
First, we show that 
\begin{equation}  
\frac{1}{B_{n+2}^2+B_{n+1}^2}+\frac{1}{B_n^2}-\frac{1}{B_{n+1}^2}<\frac{1}{B_n^2+B_{n-1}^2}\,.
\label{eq30} 
\end{equation}  
Since $B_n^2-B_{n-1}B_{n+1}=1$, we have  
\begin{align*}  
&\left(\frac{1}{B_{n+1}^2}-\frac{1}{B_{n+2}^2+B_{n+1}^2}\right)-\left(\frac{1}{B_n^2}-\frac{1}{B_n^2+B_{n-1}^2}\right)\\
&=\frac{B_{n+2}^2}{B_{n+1}^2(B_{n+2}^2+B_{n+1}^2)}-\frac{B_{n-1}^2}{B_n^2(B_n^2+B_{n-1}^2)}\\
&=\frac{1}{B_{n+1}^2+\frac{(B_{n+2}B_n+1)^2}{B_{n+2}^2}}-\frac{1}{B_n^2+\frac{(B_{n+1}B_{n-1}+1)^2}{B_{n-1}^2}}\\ 
&=\frac{1}{B_{n+1}^2+B_n^2+\frac{2 B_n}{B_{n+2}}+\frac{1}{B_{n+2}^2}}-\frac{1}{B_{n+1}^2+B_n^2+\frac{2 B_{n+1}}{B_{n-1}}+\frac{1}{B_{n-1}^2}}\\ 
&>0\,. 
\end{align*}  
Therefore, repeating (\ref{eq30}), we obtain 
\begin{equation} 
\frac{1}{B_n^2+B_{n-1}^2}>\sum_{i=0}^\infty\left(\frac{1}{B_{n+2i}^2}-\frac{1}{B_{n+2i+1}^2}\right)\,.  
\label{eq35} 
\end{equation}   
Next, we prove that 
\begin{equation}   
\frac{1}{B_n^2+B_{n-1}^2+1}<\frac{1}{B_n^2}-\frac{1}{B_{n+1}^2}+\frac{1}{B_{n+2}^2+B_{n+1}^2+1}\,. 
\label{eq40}
\end{equation}  
Since 
$$
\frac{B_{n+1}^2-2 B_{n+1}B_{n-1}-1}{B_{n-1}^2+1}>0>\frac{B_{n}^2-2 B_{n+2}B_{n}-1}{B_{n+2}^2+1}\,, 
$$ 
we have  
\begin{align*}  
&\left(\frac{1}{B_n^2}-\frac{1}{B_n^2+B_{n-1}^2+1}\right)-\left(\frac{1}{B_{n+1}^2}-\frac{1}{B_{n+2}^2+B_{n+1}^2+1}\right)\\
&=\frac{1}{B_n^2\left(\frac{B_n^2}{B_{n-1}^2+1}+1\right)}-\frac{1}{B_{n+1}^2\left(\frac{B_{n+1}^2}{B_{n+2}^2+1}+1\right)}\\
&=\frac{1}{\frac{(B_{n+1}B_{n-1}+1)^2}{B_{n-1}^2+1}+B_n^2}-\frac{1}{\frac{(B_{n+2}B_{n}+1)^2}{B_{n+2}^2+1}+B_{n+1}^2}\\  
&=\frac{1}{B_{n+1}^2+B_n^2-\frac{B_{n+1}^2-2 B_{n+1}B_{n-1}-1}{B_{n-1}^2+1}}-\frac{1}{B_{n+1}^2+B_n^2-\frac{B_{n}^2-2 B_{n+2}B_{n}-1}{B_{n+2}^2+1}}\\
&>0\,. 
\end{align*}  
Therefore, repeating (\ref{eq20}), we obtain  
\begin{equation} 
\frac{1}{B_n^2+B_{n-1}^2+1}<\sum_{i=0}^\infty\left(\frac{1}{B_{n+2i}^2}-\frac{1}{B_{n+2i+1}^2}\right)\,.  
\label{eq45} 
\end{equation}  
Combining (\ref{eq35}) and (\ref{eq45}), we get 
$$
\frac{1}{B_n^2+B_{n-1}^2+1}<\sum_{k=n}^\infty\frac{(-1)^k}{B_k^2}<\frac{1}{B_n^2+B_{n-1}^2}
$$ 
if $n$ is even, and 
$$
-\frac{1}{B_n^2+B_{n-1}^2+1}>\sum_{k=n}^\infty\frac{(-1)^k}{B_k^2}>-\frac{1}{B_n^2+B_{n-1}^2}
$$ 
if $n$ is odd. 
\end{proof}

\begin{theorem}  
$$ 
\fl{\left(\sum_{k=n}^\infty\frac{(-1)^k}{B_{2 k}}\right)^{-1}}=
\begin{cases}  
B_{2 n}+B_{2 n-2}&\text{if $n$ is even};\\ 
-(B_{2 n}+B_{2 n-2}+1)&\text{if $n$ is odd}. 
\end{cases}  
$$ 
\end{theorem} 
\begin{proof}  
First, we prove 
\begin{equation}  
\frac{1}{B_{2 n}}-\frac{1}{B_{2 n+2}}+\frac{1}{B_{2 n+4}+B_{2 n+2}}<\frac{1}{B_{2 n}+B_{2 n-2}}\,. 
\label{45}
\end{equation} 
Notice that by the recurrence relation $B_n=6 B_{n-1}-B_{n-2}$, we have 
\begin{align} 
B_n B_{n-4}-B_{n-1}B_{n-3}&=B_{n-1}B_{n-5}-B_{n-2}B_{n-4}\notag\\
&=\cdots\notag\\
&=B_4 B_0-B_3 B_1=-35\,. 
\label{rec-2}
\end{align} 
Since 
\begin{align*} 
&B_{2 n+2}\left(\frac{B_{2 n+2}}{B_{2 n+4}}+1\right)=B_{2 n+2}+\frac{B_{2 n+3}B_{2 n+1}+1}{B_{2 n+4}}\\
&=B_{2 n+2}+B_{2 n}+\frac{-B_{2 n+4}B_{2 n}+B_{2 n+3}B_{2 n+1}+1}{B_{2 n+4}}\\
&=B_{2 n+2}+B_{2 n}+\frac{36}{B_{2 n+4}}
\end{align*} 
and 
\begin{align*} 
&B_{2 n}\left(\frac{B_{2 n}}{B_{2 n-2}}+1\right)=\frac{B_{2 n+1}B_{2 n-1}+1}{B_{2 n-2}}+B_{2 n}\\
&=B_{2 n+2}+B_{2 n}+\frac{-B_{2 n+2}B_{2 n-2}+B_{2 n+1}B_{2 n-1}+1}{B_{2 n-2}}\\
&=B_{2 n+2}+B_{2 n}+\frac{36}{B_{2 n-2}}\,, 
\end{align*} 
we have 
\begin{align*} 
&\left(\frac{1}{B_{2 n+2}}-\frac{1}{B_{2 n+4}+B_{2 n+2}}\right)-\left(\frac{1}{B_{2 n}}-\frac{1}{B_{2 n}+B_{2 n-2}}\right)\\ 
&=\frac{1}{B_{2 n+2}\left(\frac{B_{2 n+2}}{B_{2 n+4}}+1\right)}-\frac{1}{B_{2 n}\left(\frac{B_{2 n}}{B_{2 n-2}}+1\right)}\\
&=\frac{1}{B_{2 n+2}+B_{2 n}+\frac{36}{B_{2 n+4}}}-\frac{1}{B_{2 n+2}+B_{2 n}+\frac{36}{B_{2 n-2}}}\\ 
&>0\,,
\end{align*}
yielding (\ref{45}).  

Next, we prove that 
\begin{equation}  
\frac{1}{B_{2 n}+B_{2 n-2}+1}<\frac{1}{B_{2 n}}-\frac{1}{B_{2 n+2}}+\frac{1}{B_{2 n+4}+B_{2 n+2}+1}\,. 
\label{56} 
\end{equation}  
By (\ref{rec-2}), we get 
\begin{align*} 
&B_{2 n}\left(\frac{B_{2 n}}{B_{2 n-2}+1}+1\right)=\frac{B_{2 n+1}B_{2 n-1}+1}{B_{2 n-2}+1}+B_{2 n}\\
&=B_{2 n+2}+B_{2 n}-\frac{B_{2 n+2}B_{2 n-2}-B_{2 n+1}B_{2 n-1}+B_{2 n+2}-1}{B_{2 n-2}+1}\\
&=B_{2 n+2}+B_{2 n}-\frac{B_{2 n+2}-36}{B_{2 n-2}+1}
\end{align*} 
and 
\begin{align*} 
&B_{2 n+2}\left(\frac{B_{2 n+2}}{B_{2 n+4}+1}+1\right)=B_{2 n+2}+\frac{B_{2 n+3}B_{2 n+1}+1}{B_{2 n+4}+1}\\
&=B_{2 n+2}+B_{2 n}-\frac{B_{2 n+4}B_{2 n}-B_{2 n+3}B_{2 n+1}+B_{2 n}-1}{B_{2 n+4}+1}\\
&=B_{2 n+2}+B_{2 n}-\frac{B_{2 n}-36}{B_{2 n+4}+1}\,. 
\end{align*} 
Since 
$$
\frac{B_{2 n}-36}{B_{2 n+4}+1}<\frac{B_{2 n+2}-36}{B_{2 n-2}+1}\,, 
$$
we obtain 
\begin{align*} 
&\left(\frac{1}{B_{2 n}}-\frac{1}{B_{2 n}+B_{2 n-2}+1}\right)-\left(\frac{1}{B_{2 n+2}}-\frac{1}{B_{2 n+4}+B_{2 n+2}+1}\right)\\ 
&=\frac{1}{B_{2 n}\left(\frac{B_{2 n}}{B_{2 n-2}+1}+1\right)}-\frac{1}{B_{2 n+2}\left(\frac{B_{2 n+2}}{B_{2 n+4}+1}+1\right)}\\
&>0\,,
\end{align*}
yielding (\ref{56}).  
Repeating applying (\ref{45}) and (\ref{56}), we have for even $n$ 
$$
\frac{1}{B_{2 n}+B_{2 n-2}+1}<\sum_{k=n}^\infty\frac{(-1)^k}{B_{2 k}}<\frac{1}{B_{2 n}+B_{2 n-2}}\,. 
$$ 
\end{proof}

The following odd case can be proven similarly, so the proof is omitted.  

\begin{theorem}  
$$ 
\fl{\left(\sum_{k=n}^\infty\frac{(-1)^k}{B_{2 k+1}}\right)^{-1}}=
\begin{cases}  
B_{2 n+1}+B_{2 n-1}&\text{if $n$ is even};\\ 
-(B_{2 n+1}+B_{2 n-1}+1)&\text{if $n$ is odd}. 
\end{cases}  
$$ 
\end{theorem}

If two consecutive balancing numbers appear, we have the following result. 

\begin{theorem}  
$$ 
\fl{\left(\sum_{k=n}^\infty\frac{(-1)^k}{B_{k}B_{k+1}}\right)^{-1}}=
\begin{cases}  
B_{n}B_{n+1}+B_{n-1}B_n&\text{if $n$ is even};\\ 
-(B_{n}B_{n+1}+B_{n-1}B_n+1)&\text{if $n$ is odd}. 
\end{cases} 
$$ 
\end{theorem} 
\begin{proof}  
First, we prove that 
\begin{multline}  
\frac{1}{B_n B_{n+1}}-\frac{1}{B_{n+1}B_{n+2}}+\frac{1}{B_{n+2}B_{n+3}+B_{n+1}B_{n+2}}\\
<\frac{1}{B_n B_{n+1}+B_{n-1}B_n}\,. 
\label{200}
\end{multline} 
By using $B_n^2=B_{n-1}B_{n+1}+1$, we have 
\begin{align*} 
&\left(\frac{1}{B_{n+1}B_{n+2}}-\frac{1}{B_{n+2}B_{n+3}+B_{n+1}B_{n+2}}\right)\\
&\quad -\left(\frac{1}{B_{n}B_{n+1}}-\frac{1}{B_{n}B_{n+1}+B_{n-1}B_{n}}\right)\\
&=\frac{1}{B_{n+1}B_{n+2}\left(1+\frac{B_{n+1}}{B_{n+3}}\right)}-\frac{1}{B_{n}B_{n+1}\left(1+\frac{B_{n+1}}{B_{n-1}}\right)}\\
&=\frac{1}{B_{n+1}B_{n+2}+(1+B_n B_{n+2})\frac{B_{n+2}}{B_{n+3}}}-\frac{1}{B_{n}B_{n+1}+(1+B_n B_{n+2})\frac{B_{n}}{B_{n-1}}}\\
&=\frac{1}{B_{n+1}B_{n+2}+B_{n}B_{n+1}+\frac{B_n+B_{n+2}}{B_{n+3}}}-\frac{1}{B_{n+1}B_{n+2}+B_{n}B_{n+1}+\frac{B_n+B_{n+2}}{B_{n-1}}}\\
&>0
\end{align*} 
since 
$$
\frac{B_n+B_{n+2}}{B_{n+3}}<\frac{B_n+B_{n+2}}{B_{n-1}}\,. 
$$ 
Second, we show that 
\begin{multline}  
\frac{1}{B_n B_{n+1}+B_{n-1}B_n+1}\\
<\frac{1}{B_n B_{n+1}}-\frac{1}{B_{n+1}B_{n+2}}+\frac{1}{B_{n+2}B_{n+3}+B_{n+1}B_{n+2}+1}\,. 
\label{205} 
\end{multline}  
We have 
\begin{align*}  
&\left(\frac{1}{B_n B_{n+1}}-\frac{1}{B_n B_{n+1}+B_{n-1}B_n+1}\right)\\
&\quad -\left(\frac{1}{B_{n+1}B_{n+2}}-\frac{1}{B_{n+2}B_{n+3}+B_{n+1}B_{n+2}+1}\right)\\
&=\frac{1}{B_n B_{n+1}\left(1+\frac{B_n B_{n+1}}{B_{n-1}B_n+1}\right)}-\frac{1}{B_{n+1}B_{n+2}\left(1+\frac{B_{n+1}B_{n+2}}{B_{n+2}B_{n+3}+1}\right)}\\
&=\frac{1}{B_n B_{n+1}+\frac{(1+B_{n-1}B_{n+1})(1+B_n B_{n+2})}{B_{n-1}B_n+1}}\\
&\quad -\frac{1}{B_{n+1}B_{n+2}+\frac{(1+B_{n}B_{n+2})(1+B_{n+1}B_{n+3})}{B_{n+2}B_{n+3}+1}}\\
&=\frac{1}{B_n B_{n+1}+B_{n+1}B_{n+2}+\frac{1+B_{n-1}B_{n+1}+B_n B_{n+2}-B_{n+1}B_{n+2}}{B_{n-1}B_n+1}}\\
&\quad -\frac{1}{B_n B_{n+1}+B_{n+1}B_{n+2}+\frac{1+B_{n}B_{n+2}+B_{n+1}B_{n+3}-B_{n}B_{n+1}}{B_{n+2}B_{n+3}+1}}\\
&>0 
\end{align*} 
since 
\begin{multline*}  
1+B_{n-1}B_{n+1}+B_n B_{n+2}-B_{n+1}B_{n+2}<0\\
<1+B_{n}B_{n+2}+B_{n+1}B_{n+3}-B_{n}B_{n+1}\,. 
\end{multline*}  
By repeating (\ref{200}) and (\ref{205}), respectively, we have 
$$
\frac{1}{B_n B_{n+1}+B_{n-1}B_n+1}<\sum_{k=n}^\infty\frac{(-1)^k}{B_k B_{k+1}}<\frac{1}{B_n B_{n+1}+B_{n-1}B_n}\,. 
$$  
\end{proof}

Similarly, we can have the following results. 

\begin{theorem}  
$$ 
\fl{\left(\sum_{k=n}^\infty\frac{(-1)^k}{B_{2 k}^2}\right)^{-1}}=
\begin{cases}  
B_{2 n}^2+B_{2 n-2}^2&\text{if $n$ is even};\\ 
-(B_{2 n}^2+B_{2 n-2}^2+1)&\text{if $n$ is odd}. 
\end{cases}  
$$ 
$$ 
\fl{\left(\sum_{k=n}^\infty\frac{(-1)^k}{B_{2 k-1}^2}\right)^{-1}}=
\begin{cases}  
B_{2 n-1}^2+B_{2 n-3}^2&\text{if $n$ is even};\\ 
-(B_{2 n-1}^2+B_{2 n-3}^2+1)&\text{if $n$ is odd}. 
\end{cases}  
$$ 
$$ 
\fl{\left(\sum_{k=n}^\infty\frac{(-1)^k}{B_{2 k-1}B_{2 k+1}}\right)^{-1}}=
\begin{cases}  
B_{2 n}^2+B_{2 n-2}^2-1&\text{if $n$ is even};\\ 
-(B_{2 n}^2+B_{2 n-2}^2)&\text{if $n$ is odd}. 
\end{cases}  
$$ 
$$ 
\fl{\left(\sum_{k=n}^\infty\frac{(-1)^k}{B_{2 k}B_{2 k+2}}\right)^{-1}}=
\begin{cases}  
B_{2 n+1}^2+B_{2 n-1}^2-1&\text{if $n$ is even};\\ 
-(B_{2 n+1}^2+B_{2 n-1}^2)&\text{if $n$ is odd}. 
\end{cases}  
$$ 
\end{theorem}

For Lucas-balancing numbers $C_n$ we have the following corresponding results.  
Notice that $C_n$ satisfies the recurrence relation 
$C_n=6 C_{n-1}-C_{n-2}$ ($n\ge 2$) with $C_0=1$ and $C_1=3$.  

\begin{theorem}  
$$ 
\fl{\left(\sum_{k=n}^\infty\frac{(-1)^k}{C_k}\right)^{-1}}=
\begin{cases}  
C_n+C_{n-1}-1&\text{if $n$ is even};\\ 
-(C_n+C_{n-1})&\text{if $n$ is odd}. 
\end{cases} 
$$ 
$$ 
\fl{\left(\sum_{k=n}^\infty\frac{(-1)^k}{C_k^2}\right)^{-1}}=
\begin{cases}  
C_n^2+C_{n-1}^2-1&\text{if $n$ is even};\\ 
-(C_n^2+C_{n-1}^2)&\text{if $n$ is odd}. 
\end{cases} 
$$ 
$$ 
\fl{\left(\sum_{k=n}^\infty\frac{(-1)^k}{C_{2 k}}\right)^{-1}}=
\begin{cases}  
C_{2 n}+C_{2 n-2}-1&\text{if $n$ is even};\\ 
-(C_{2 n}+C_{2 n-2})&\text{if $n$ is odd}. 
\end{cases}  
$$ 
$$ 
\fl{\left(\sum_{k=n}^\infty\frac{(-1)^k}{C_{2 k+1}}\right)^{-1}}=
\begin{cases}  
C_{2 n+1}+C_{2 n-1}-1&\text{if $n$ is even};\\ 
-(C_{2 n+1}+C_{2 n-1})&\text{if $n$ is odd}. 
\end{cases}  
$$ 
$$ 
\fl{\left(\sum_{k=n}^\infty\frac{(-1)^k}{C_{k}C_{k+1}}\right)^{-1}}=
\begin{cases}  
C_{n}C_{n+1}+C_{n-1}C_n-1&\text{if $n$ is even};\\ 
-(C_{n}C_{n+1}+C_{n-1}C_n)&\text{if $n$ is odd}. 
\end{cases} 
$$ 
$$ 
\fl{\left(\sum_{k=n}^\infty\frac{(-1)^k}{C_{2 k}^2}\right)^{-1}}=
\begin{cases}  
C_{2 n}^2+C_{2 n-2}^2-1&\text{if $n$ is even};\\ 
-(C_{2 n}^2+C_{2 n-2}^2)&\text{if $n$ is odd}. 
\end{cases}  
$$ 
$$ 
\fl{\left(\sum_{k=n}^\infty\frac{(-1)^k}{C_{2 k-1}^2}\right)^{-1}}=
\begin{cases}  
C_{2 n-1}^2+C_{2 n-3}^2-1&\text{if $n$ is even};\\ 
-(C_{2 n-1}^2+C_{2 n-3}^2)&\text{if $n$ is odd}. 
\end{cases}  
$$ 
$$ 
\fl{\left(\sum_{k=n}^\infty\frac{(-1)^k}{C_{2 k-1}C_{2 k+1}}\right)^{-1}}=
\begin{cases}  
C_{2 n}^2+C_{2 n-2}^2-1&\text{if $n$ is even};\\ 
-(C_{2 n}^2+C_{2 n-2}^2)&\text{if $n$ is odd}. 
\end{cases}  
$$ 
$$ 
\fl{\left(\sum_{k=n}^\infty\frac{(-1)^k}{C_{2 k}C_{2 k+2}}\right)^{-1}}=
\begin{cases}  
C_{2 n+1}^2+C_{2 n-1}^2-1&\text{if $n$ is even};\\ 
-(C_{2 n+1}^2+C_{2 n-1}^2)&\text{if $n$ is odd}. 
\end{cases}  
$$ 
\end{theorem}


\end{document}